\documentclass[article,12pt]{elsarticle}

	\usepackage[margin=1.5in]{geometry}  % set the margins to 1in on all sides
	\usepackage{graphicx}              % to include figures
	\usepackage{amsmath} 
	\usepackage{orcidlink}              % great math stuff
	\usepackage{amsfonts}
	\usepackage{amsthm} 
	\usepackage{colortbl}               % better theorem environments
	\usepackage{framed}
	\usepackage{comment}
	\usepackage{url}
	\newtheorem{thm}{Theorem}[section]
	\newtheorem{lemma}[thm]{Lemma}
	\newtheorem{prop}[thm]{Proposition}
	\newtheorem{cor}[thm]{Corollary}
	
	\newtheorem{dfn}[thm]{Definition}
	\newtheorem{rmk}[thm]{Remark}

	\makeatletter
	\def\ps@pprintTitle{%
		\let\@oddhead\@empty
		\let\@evenhead\@empty
		\def\@oddfoot{\reset@font\hfil}
		\def\@evenfoot{\reset@font\hfil}
	}
	\makeatother
\begin{document}
		
		\begin{frontmatter}
			
			\title{\bf On the Order Estimates for Specific Functions of $\zeta(s)$ and its Contribution towards the Analytic Proof of The Prime Number Theorem}
			
			\author[affil]{\textsc{Subham De} \orcidlink{0009-0001-3265-4354}}
			
			\address[affil]{Department of Mathematics, Indian Institute of Technology Delhi, India \footnote{email: subham581994@gmail.com}\footnote{Website: \url{www.sites.google.com/view/subhamde}}}
			
			\begin{abstract}
				\noindent This article provides a proof of the famous \textit{Prime Number Theorem} by establishing an analogous statement of the same in terms of the second \textit{Chebyshev Function} $\psi(x)$. We shall be extensively using complex analytic techniques in addition to certain meromorphic properties of the \textit{Reimann Zeta Function} $\zeta(s)$ and its \textit{Analytic Continuation Property} a priori using Riemann's Functional Equation in order to establish our desired result.
			\end{abstract}
			
			\begin{keyword}
				Prime Number Theorem, Riemann Zeta Function, Holomorphic, Contour Integral, Critical Line
				
				\MSC[2020] Primary  11-02, 11A25, 11M26, 11N37 \sep Secondary 11-03, 11M06, 30D30, 32E20
			\end{keyword}

		\end{frontmatter}
		
	\section{Introduction}	
	In Number Theory, the \textbf{Prime Number Theorem (PNT)} describes the asymptotic distribution of the prime numbers among the positive integers. It formalizes the intuitive idea that primes become less common as they become larger by precisely quantifying the rate at which this occurs. The theorem was proved independently by \textit{Jacques Hadamard} and \textit{Charles Jean de la Vallée-Poussin} in $1896$ using ideas introduced by \textit{Bernhard Riemann} (in particular, the \textit{Riemann zeta function}).\par
	The first such distribution found is $\pi(x)\sim \frac{x}{\log x}$, where $\pi(x)$ is the prime-counting function and $\log x$ is the natural logarithm of $x$. This means that for large enough $x$, the probability that a random integer $n$ not greater than $x$ is prime is very close to $\frac{1}{\log x}$. Consequently, a random integer with at most $2k$ digits (for large enough $k$) is about half as likely to be prime as a random integer with at most $k$ digits. For example, among the positive integers of at most $1000$ digits, about one in $2300$ is prime $(\log 101000\approx 2302.6)$, whereas among positive integers of at most $2000$ digits, about one in $4600$ is prime $(\log 102000 \approx 4605.2)$. In other words, the average gap between consecutive prime numbers among the first $n$ integers is roughly $\log n $.
	\section{Statement of the Prime Number Theorem}
	\subsection{Notion of Arithmetic Functions}
	We first introduce some special arithmetic functions and notations required to appreciate the analytic aspects of the \textit{Prime Number Theorem} \cite{6}.
	\begin{dfn}\label{def1}
		For each $\textit{x}\geq0$ ,we define,
		\begin{center}
			$ \pi(x):=$The number of primes $\leq \textit{x}$.
		\end{center}
	\end{dfn}
	\begin{dfn}\label{def2}
		For each $\textit{x}\geq0$ , we define,
		\begin{center}
			$\psi(x):=\sum\limits_{n\le x}\Lambda(n)$ ,
		\end{center}
		Where ,  
		\begin{eqnarray}\label{1}
			\Lambda(n) 
			:=\left\{
			\begin{array}{cc}
				\log p\mbox{ }, &\mbox{ if }n=p^m,\mbox{ } p^m\leq x,\mbox{ }m\in \mathbb{N}\\
				0\mbox{ }, &\mbox{ otherwise }.
			\end{array}
			\right.
		\end{eqnarray} 
		$ \Lambda(n)$ is said to be the "\textit{Mangoldt Function}"  .\\
		Therefore,
		\begin{eqnarray}\label{2}
			\psi(x) = \sum\limits_{n\leq x}\Lambda(n) = \sum\limits_{m=1}^{\infty}\sum\limits_{p ,  p^m\leq x} \Lambda(p^m) = \sum\limits_{m=1}^{\infty}\sum\limits_{p\leq x^{\frac{1}{m}}} \log p
		\end{eqnarray}
	\end{dfn} 
	\begin{dfn}\label{def3}
		(\textbf{Chebyshev Theta Function}) For each $x\geq0$, we define,
		\begin{center}
			$ \vartheta(x) := \sum\limits_{p\leq x} \log p $.
		\end{center}  
	\end{dfn} 
	\begin{rmk}\label{rmk1}
		A priori from definitions \eqref{def2} and \eqref{def3}, it can be deduced that,
		\begin{center}
			$ \psi(x) = \sum\limits_{m\leq log_{2}{x}} \vartheta(x^{\frac{1}{m}}) $
		\end{center} 
	\end{rmk}

	\begin{dfn}\label{def4}
		(\textbf{M\"{o}bius Function})  The M\"{o}bius Function $ \mu $ is defined as follows :
		\begin{center}
			$\mu(1) = 1$.
		\end{center}
		
		If $ n>1 $, such that suppose, $ n = {{p_{1}}^{a_{1}}}{{p_{2}}^{a_{2}}}{{p_{3}}^{a_{3}}}\cdots{{p_{k}}^{a_{k}}} $ . Then,
		\begin{eqnarray}\label{3}
			\mu(n) :=\left\{
			\begin{array}{cc}
				(-1)^{k}\mbox{ }, &\mbox{ if }a_{1} = a_{2} = \cdots\cdots = a_{k} = 1\\
				0\mbox{ }, &\mbox{ otherwise }.
			\end{array}
			\right.
		\end{eqnarray}
	\end{dfn} 
	\begin{dfn}\label{def5}
		(Big $O$ Notation)  Given $g(x) > 0\hspace{10pt} \forall \mbox{  }x\geq a$, the notation, $f(x) = O(g(x))$ implies that, the quotient, $\frac{f(x)}{g(x)}$ is bounded for all $x\geq a$;  i.e., $\exists$ a constant $M > 0$ such that,
		\begin{center}
			$\arrowvert f(x)\arrowvert \leq M.g(x)\hspace{10pt}\mbox{,  }\forall \mbox{  }x\geq a$ .
		\end{center}
	\end{dfn}
	\subsection{Asymptotic Bound for $\pi(x)$}
	\begin{dfn}\label{def6}
		We say $f(x)$ is \textit{asymptotic} to $g(x)$, and denote it by, $f(x) \sim g(x)$ if, \space\space  $\lim\limits_{x\to\infty}\frac{f(x)}{g(x)}= 1$.
	\end{dfn} 
	A priori having all the necessary notations and definitions in our arsenal, we can formally state the \textit{"Prime Number Theorem"}.
	\begin{thm}\label{thm1}
		(\textbf{Prime Number Theorem})  For every real number $x\geq 0$, we define the \textit{prime counting function} $\pi(x)$ as in \eqref{def1}. Then, the following estimate is valid.  
		\begin{align}
			\pi(x) \sim \frac{x}{\log x} 
		\end{align}
		Equivalently, 
		\begin{align}
			\lim\limits_{x\to\infty}\frac{\pi(x)\log x}{x} = 1 
		\end{align} . 
	\end{thm} 
	Given the difficulty in establishing the above result in support of proving the prime number theorem, we intend to explore about the possibility of working on an equivalent version of \textit{Theorem} \eqref{thm1}, which obviously will be less tedious and easy to comprehend.
	\begin{thm}\label{thm2}
		We have,
		\begin{align}
			\vartheta(x) \sim \pi(x)\log x 
		\end{align}
	\end{thm}
	\begin{proof}
		A priori it follows from definition and the Theorem of Partial Sums of Dirichlet Product that,
		\begin{align}
			\vartheta(x) = \sum\limits_{x^{1-\varepsilon}\leq p\leq x} \log p\mbox{ }\lfloor\frac{\log x}{\log p}\rfloor \leq \sum\limits_{p\leq x} \log x \leq \pi(x)\log x 
		\end{align}

		Hence, for arbitrary $ \varepsilon>0 $ ,
		\begin{align*}
			\vartheta(x)\geq \sum\limits_{x^{1-\varepsilon}\leq p\leq x} \log p \geq \sum\limits_{x^{1-\varepsilon}\leq p\leq x} (1-\varepsilon)\log x = (1-\varepsilon)(\pi(x) +O(x^{1- \varepsilon}))\log x  
		\end{align*}
		
		Therefore, we conclude, 
		\begin{align}
			\Rightarrow \vartheta(x) \rightarrow \pi(x)\log x \mbox{, \space  as \hspace{10pt}}  x \rightarrow \infty 
		\end{align}
		\begin{align}
			\Rightarrow \vartheta(x) \sim \pi(x)\log x  .
		\end{align}
	\end{proof}
	An application of \textit{Theorem \eqref{thm2}} along with properties of the functions $\vartheta(x)$ and $\psi(x)$ we can indeed formulate an \textit{alternative statement of Prime Number Theorem} :
	\begin{thm}\label{thm3}
		\begin{align}
			\psi(x) \sim x \mbox{\hspace{20pt}  as \hspace{10pt}} x \rightarrow \infty.
		\end{align}  
	\end{thm} 
	Hence, in order to establish the \textbf{Prime Number Theorem}, it only suffices to show that, statement \eqref{thm3} hods true.
	\section{A brief survey on Complex Analysis}
	
	In this section, we shall recall some important results in Complex Analysis, pertinent to the proof. Readers are encouraged to refer to \cite{2} for further explanations in these topics.
	\begin{thm}\label{thm4}
		(\textbf{Riemann-Lebesgue Lemma})  If $f$ is $ L^{1}$ integrable on $\mathbb{R}^{d}$, i.e. to say, if the Lebesgue integral of $|f|$ is finite, then the Fourier transform of $f$ satisfies,
		\begin{center}
			$\hat{f}(z) = \int\limits_{\mathbb{R}^{d}} f(x) e^{-2\pi iz.x} dx \rightarrow 0$  as,   $ |z| \rightarrow \infty .$
		\end{center}
	\end{thm}
	\begin{thm}\label{thm5}
		(Cauchy Integral Theorem)  In mathematics, the \textbf{Cauchy Integral Theorem} (also known as the \textit{Cauchy-Goursat theorem}) in complex analysis, named after \textit{Augustin-Louis Cauchy}, is an important statement about line integrals for holomorphic functions in the complex plane. Essentially, it says that if two different paths connect the same two points, and a function is holomorphic everywhere "in between" the two paths, then the two path integrals of the function will be the same.\\
		The theorem is usually formulated for closed paths as follows:
		\begin{center}
			Let $U$ be an open subset of $\mathbb{C}$ which is simply connected, let $f : U \rightarrow \mathbb{C}$ be a \textit{holomorphic function}, and let $\gamma$ be a rectifiable path in $U$ whose start point is equal to its end point. Then,
			\begin{center}
				$\oint\limits_{\gamma}f(z)dz=0$.
			\end{center}
		\end{center}
	\end{thm}
	A precise \textit{(homology)} version can be stated using \textit{winding numbers}. The winding number of a closed curve around a point $a$ not on the curve is the integral of $\frac{f(z)}{2\pi i}$, where $f(z) = \frac{1}{(z-a)}$ around the curve. It is an integer. Briefly, the path integral along a Jordan curve of a function holomorphic in the interior of the curve, is $0$. Instead of a single closed path we can consider a linear combination of closed paths, where the scalars are integers. Such a combination is called a \textit{closed chain}, and one defines an integral along the chain as a linear combination of integrals over individual paths. A closed chain is called a \textit{cycle} in a region, if it is homologous to $0$ in the region; that is, the winding number, expressed by the integral of $\frac{1}{(z-a)}$ over the closed chain, is $0$ for each point $a$ not in the region. This means that the closed chain does not wind around points outside the region. Then \textbf{Cauchy's Theorem} can be stated as the integral of a function holomorphic in an open set taken around any cycle in the open set is $0$. An example is furnished by the \textbf{ring-shaped region}. This version is crucial for rigorous derivation of \textit{Laurent Series} and \textit{Cauchy's Residue Formula} without involving any physical notions such as cross cuts or deformations. This version enables the extension of Cauchy's theorem to multiply-connected regions analytically.
	
	\begin{thm}\label{thm6}
		(Cauchy Residue Theorem)  In complex analysis, the \textbf{Residue Theorem}, sometimes called \textbf{Cauchy's Residue Theorem}, is a powerful tool to evaluate line integrals of \textit{analytic functions} over closed curves; it can often be used to compute real integrals as well. It generalizes the \textit{Cauchy Integral Theorem} and \textit{Cauchy's Integral Formula}. From a geometrical perspective, it is a special case of the generalized \textbf{Stokes' Theorem}. The statement goes as follows:\\\\
		Suppose $U$ is a \textit{simply connected open subset} of the complex plane, and $a_{1},a_{2},\cdots a_{n}$ are finitely many points of $U$ and $f$ is a function which is defined and holomorphic on $U \slash {a_{1},a_{2},\cdots a_{n}}$. If $\gamma$ is a \textit{closed rectifiable curve} in $U$ which does not meet any of the $a_{k}$,
		\begin{center}
			$\oint\limits_{\gamma}f(z)dz=2\pi i\sum\limits_{k=1}^{n}\eta(\gamma, a_{k})Res(f;a_{k})$.
		\end{center}
		
		If $\gamma$ is a positively oriented simple closed curve, $\eta(\gamma, a_{k}) = 1$ if $a_{k}$ is in the interior of $\gamma$, and $0$ otherwise, so,
		\begin{center}
			$\oint\limits_{\gamma}f(z)dz=2\pi i\sum\limits_{k=1}^{n}Res(f;a_{k})$.
		\end{center}
		with the sum over those $k$ for which $a_{k}$ is inside $\gamma$.
	\end{thm}
	Here, $Res(f;a_{k})$ denotes the residue of $f$ at $a_{k}$, and $\eta(\gamma, a_{k})$ is the winding number of the curve $\gamma$ about the point $a_{k}$. This \textit{winding number} is an integer which intuitively measures how many times the curve $\gamma$ winds around the point $a_{k}$; it is positive if $\gamma$ moves in a counter clockwise (mathematically positive) manner around $a_{k}$ and $0$ if $\gamma$ doesn't move around $a_{k}$ at all.
	\vspace{10pt}
	\section{Some Important Results}
	\vspace{10pt}
	\begin{lemma}\label{lemma1}
		Given any arithmetic function $a(n)$, suppose,
		\begin{align*}
			A(x)=\sum\limits_{n\leq x}a(n),
		\end{align*}	
		where, $A(x)=0$, if $x<1$. Then,
		\begin{eqnarray}\label{4}
			\sum\limits_{n\leq x}(x-a)a(n)=\int\limits_{1}^{x}A(t)dt.
		\end{eqnarray}
	\end{lemma}
	
	We apply \textit{Abel's Identity} to prove the result.
	\begin{thm}
		\begin{eqnarray}\label{5}
			\sum\limits_{n\leq x}a(n)f(n)=A(x)f(x)-\int\limits_{1}^{x}A(t)f'(t)dt 
		\end{eqnarray}
		
		provided $f$ has a continuous derivative on $[1,x]$.
	\end{thm}
	\begin{proof}
		Choosing $f(t)=t$, we obtain,
		\begin{align*}
			\sum\limits_{n\leq x}a(n)f(n)=\sum\limits_{n\leq x}na(n)\mbox{ \hspace{10pt}and, \hspace{10pt}}A(x)f(x)=n\sum\limits_{n\leq x}a(n).
		\end{align*}
		Thus the result follows from \textit{Theorem} \eqref{5}.
	\end{proof}
	In order to estimate $A(x)$ as in above, we require the following lemma can also be regarded as a form of \textit{L'H\`{o}spital's} rule for increasing piece-wise linear functions.
	\begin{lemma}\label{lemma2}
		Let $A(x)=\sum\limits_{n\leq x}a(n)$ and let $A_{1}(x)=\int\limits_{1}^{x}A(t)dt$. Assume also that, $a(n)\geq 0$,       $\forall$  $ n\in \mathbb{N}$. Given the asymptotic relation ,
		\begin{eqnarray}\label{6}
			A_{1}(x)\sim Lx^{c}    \mbox{    as     } x\rightarrow \infty
		\end{eqnarray}
		For some $c>0$ and $L>0$, we shall have,
		\begin{eqnarray}\label{7}
			A(x)\sim cLx^{c-1}       \mbox{    as   } x\rightarrow \infty
		\end{eqnarray}
		In other words, formal differentiation of \eqref{6} gives	the result \eqref{7}.
	\end{lemma}
	\begin{proof}
		The function $A(x)$ is increasing, since the $a(n)$ 's are non-negative. Consider any $\beta >1$. Thus, we get,
		\begin{align*}
			A_{1}(\beta x)-A_{1}(x)=\int\limits_{x}^{\beta x}A(u)du\geq \int\limits_{x}^{\beta x}A(x)du=A(x)(\beta (x)-x)
			=x(\beta -1)A(x).
		\end{align*}
		This implies,
		\begin{align*}
			xA(x)\leq \frac{1}{(\beta -1)}{A_{1}(\beta x)-A_{1}(x)},
		\end{align*}
		or,
		\begin{align*}\label{24}
			\frac{A(x)}{x^{c-1}}\leq \frac{1}{(\beta -1)}\{\frac{A_{1}(\beta x)}{(\beta x )^{c}}\beta ^{c}-\frac{A_{1}(x)}{x^{c}}\}.
		\end{align*}	
		Keeping $\beta $ fixed and taking $x\rightarrow \infty$ in \eqref{24}, we obtain,
		\begin{align}
			\limsup\limits_{x\rightarrow \infty}\frac{A(x)}{x^{c-1}}\leq \frac{1}{\beta -1}(L\beta ^{c}- L)=L\frac{\beta^{c}-1}{\beta -1}.
		\end{align}	
		Further, consider $\beta \rightarrow 1^{+}$. The quotient on the right is the difference quotient for the derivative of $x^{c}$ at $x=1$ and has the limit $c$. Therefore,
		\begin{eqnarray}\label{8}
			\limsup\limits_{x\rightarrow \infty}\frac{A(x)}{x^{c-1}}\leq cL.
		\end{eqnarray}
		Assume any $\alpha$ with $0<\alpha <1$ and consider the difference, $\{A_{1}(x)-A_{1}(\alpha x)\}$. By similar arguments done previously during the proof of this lemma, we get,
		\begin{align}
			\liminf\limits_{x\rightarrow \infty}\frac{A(x)}{x^{c-1}}\geq L\frac{1-\alpha ^{c}}{1-\alpha}.
		\end{align}	
		Now, as $\alpha\rightarrow 1^{-}$, the term in the R.H.S. tends to $cL$. This, together with \eqref{8} evokes that,
		\begin{align}
			\frac{A(x)}{x^{c-1}} \rightarrow cL\mbox{\hspace{10pt} as\hspace{10pt}} x\rightarrow \infty.
		\end{align} 
		And we are done.
	\end{proof}
	A priori if we assume, $a(n)=\Lambda(n)$, then $A(x)=\psi(x)$, $A_{1}(x)=\psi_{1}(x)$, and $a_{n}\geq 0$.\\
	It follows from lemmas \eqref{lemma1} and \eqref{lemma2},
	\begin{thm}\label{thm7}
		\begin{eqnarray}\label{9}
			\psi_{1}(x)=\sum\limits_{n\leq x}(x-n)\lambda(n).
		\end{eqnarray}
		Where, the \textit{Liouville's Function}, $\lambda(n)$ is defined as follows:
		\begin{center}
			$\lambda(n) 
			=\left\{
			\begin{array}{cc}
				1\mbox{ }, &\mbox{ if }n=1\\
				\hspace{30pt}(-1)^{\sum\limits_{i=1}^{k}{\alpha_{i}}} \mbox{ },  &\mbox{     if } \mbox{  } n=\prod\limits_{i=1}^{k}{p_{i}}^{\alpha_{i}}.
			\end{array}
			\right.$
		\end{center}
		Also the asymptotic relation, $\psi_{1}(x)\sim \frac{x^{2}}{2}$ implies, $\psi(x)\sim x$	as $x\rightarrow \infty$.
	\end{thm}
	
	Our next goal is to express $\frac{\psi_{1}(x)}{x^{2}}$ as a \textit{Contour Integral} involving the \textit{zeta function} ( We shall provide an overview of the \textit{Riemann Zeta Function} in the next section ). For this, we will require the special cases, $k=1$ and $k=2$ of the following lemma on contour integrals.
	
	\begin{lemma}\label{lemma3}
		If $c>0$ and $u>0$, then for every integer $k\geq 1$, we have,
		\begin{center}
			$\frac{1}{2\pi i}\int\limits_{c-\infty i}^{c+\infty i}\frac{u^{-z}}{z(z+1)(z+2)\cdots (z+k)}dz=
			\left\{
			\begin{array}{cc}
				\frac{1}{k!}(1-u)^{k}$ $, &\mbox{ if }0<u\leq 1,\\
				0$ $, &\mbox{ if }u>1.
			\end{array}
			\right. $
		\end{center}
		the integral being absolutely convergent.
	\end{lemma}
	\begin{proof}
		We observe that, the integrand is $=\frac{u^{-z}\Gamma(z)}{\Gamma(z+k+1)}$ ( This follows by repeatedly using the functional equation, $\Gamma(z+1)=z.\Gamma(z)$ ), and application of \textit{Cauchy's Residue Theorem}. ( For detailed proof of the lemma, readers can see  \cite[p.~281-282]{1} )
	\end{proof}
	\section{Analytic Continuation Property of $\zeta(s)$}
	\subsection{Motivation to study $\zeta(s)$}

	\textit{Riemann Zeta Function} \cite{1} is an integral part of \textit{Analytic Number Theory} \cite{3} \cite{4}, often treated as a special case of the \textit{Hurwitz Zeta Function} $\zeta(s,a)$, defined for $Re(s)>1$, as the series,
	\begin{center}
		$\zeta(s,a)=\sum\limits_{n=0}^{\infty}\frac{1}{(n+a)^s}$,               where, $a\in \mathbb{R}$, $0<a\leq 1$ is fixed.
	\end{center}
	\textit{Riemann Zeta Function} was first introduced by \textit{Leonhard Euler} in the first half of eighteenth century, using only \textit{Real numbers}. Also, he even computed the values of the zeta function at even positive integers.\par
	Later on, famous mathematician \textit{Bernhard Riemann} extended Euler's definition of the Riemann Zeta Function on $\mathbb{R}$ to the field of \textit{Complex Numbers}, also deriving the \textit{meromorphic continuation} and \textit{functional equation}, and gave us an idea about the relationship between the zeroes of the Riemann Zeta Function and the \textit{distribution of prime numbers}. The details can be found in his article titled \textit{On the Number of Primes less than a given Magnitude} published in $1859$. That's the reason why, afterwards, this function was named after him.\par 
	\textit{Riemann Zeta Function} is hugely significant in the field of \textit{Number Theory}. For example, $\zeta(2)$ provides solution to the famous \textit{"Basel Problem"}. Also, famous Greek-French mathematician \textit{Roger Ap\'{e}ry} proved in the year $1979$ that, $\zeta(3)$ is \textit{irrational}. \textit{Euler} established the fact that, the Riemann Zeta Function yields \textit{rational values} at the negative integer points and moreover, these values are incredibly useful in the field of \textit{Modular Forms}. Functions like, \textit{Dirichlet Series}, \textit{Dirichlet L-functions}, and \textit{L-functions} are often considered to be generalisation of the \textit{Riemann Zeta Function}.
	
	\subsection{Analytic Properties of $\zeta(s)$}
	
	Enough with its historical significance, let us give a proper definition of the \textit{Riemann Zeta Function} below.
	\begin{dfn}
		(\textbf{Riemann Zeta Function})  The \textit{Riemann Zeta Function} $\zeta(s)$ is defined as,
		\begin{center}
			$\zeta(s)=\sum\limits_{n=1}^{\infty}\frac{1}{n^{s}}$,     where,  $s\in \mathbb{C}$, with $Re(s)>1$.
		\end{center}
		And we can further extend the \textit{Riemann Zeta Function} to the whole complex plane using the \textit{Analytic Continuation Property} of the function defined for $Re(s)>1$.
	\end{dfn}
	A priori from the definition of the \textit{Gamma Function} $\Gamma(s)$ given as,
	\begin{center}
		$\Gamma(s)=\int\limits_{0}^{\infty}{x^{s-1}}{e^{-x}}dx$,  where, $s\in \mathbb{C}$,  $Re(s)>0$,
	\end{center}
	we can thus provide an alternative definition of the \textit{Riemann Zeta Function} as,
	\begin{dfn}
		We have,
		\begin{center}
			$\zeta(s)=\frac{1}{\Gamma(s)}\int\limits_{0}^{\infty}\frac{x^{s-1}}{e^{x}-1}dx$,    where, $s\in \mathbb{C}$ and,  $Re(s)>1$.
		\end{center}
	\end{dfn}
	As for some important properties of $\zeta(s)$, they can be summarized as follows:
	\begin{prop}
		$\zeta(s)$ is \textit{meromorphic} on $\mathbb{C}$, i.e., $\zeta(s)$ is \textit{holomorphic} everywhere except for a \textit{simple pole} at $s=1$ with \textit{residue} $1$.
	\end{prop}
	\begin{proof}
		Follows from the deduction that,
		\begin{center}
			$\lim_{s\rightarrow 1}(s-1).\zeta(s)=1$.
		\end{center}
	\end{proof}
	\begin{prop}
		For $s=1$, $\zeta(s)$ is the \textit{harmonic series} that \textit{diverges} to $+\infty$.
	\end{prop}
	\begin{prop}
		(\textbf{Basel Problem}) 
		\begin{center}
			$\zeta(2)=\sum\limits_{n=1}^{\infty}\frac{1}{n^{2}}=\frac{\pi^{2}}{6}$.
		\end{center}        
	\end{prop}
	\begin{prop}
		(\textbf{Trivial Zero-Free Region}) The Riemann Zeta Function $\zeta(s)$ has no zeroes in the region , $\{s\in \mathbb{C}\space :\space Re(s)>1\}$ .
	\end{prop}
	
	\subsection{Euler Product Representation of $\zeta(s)$}
	
	Deduced by famous mathematician \textit{Euler}, in the year $1737$, the following identity establishes a relation between the \textit{Riemann Zeta Function} and \textit{prime numbers}. The result is as follows :
	\begin{thm}
		(\textbf{Euler Product Formula})  We have, 
		\begin{center}
			$\zeta(s)=\sum\limits_{n=1}^{\infty}\frac{1}{n^{s}}=\prod\limits_{p=prime}\frac{1}{(1-{p^{-s}})}$
		\end{center}
		Where, the product on the R.H.S. is taken over all \textit{primes} $p$ and, converges for $Re(s)>1$.
	\end{thm}
	\begin{proof}
		We shall state a lemma in support of establishing the above result.
		\begin{lemma}\label{lemma4}
			Suppose $\{a_{n}\}_{n\in \mathbb{N}}$ be a sequence of \textit{complex numbers} with $a_{n}\neq -1$,  $\forall$ $ n\in \mathbb{N}$. Then,
			\begin{center}
				$\sum\limits_{n=1}^{\infty}{|a_{n}|}$ is convergent,     implies,     $\prod\limits_{n=1}^{\infty}{(1+a_{n})}$  is convergent.
			\end{center}
		\end{lemma}
		Assume, $\sigma=Re(s)$. Then for natural numbers $N_{1}$ and $N_{2}$ with, $N_{1}<N_{2}$ (without loss of generality), 
		\begin{center}
			$|\sum\limits_{N_{1}+1}^{N_{2}}\frac{1}{n^{s}}|\leq |\sum\limits_{N_{1}+1}^{N_{2}}\frac{1}{n^{\sigma}}|$ \hspace{40pt}  (Using the result, $|n^{-\sigma}|=|n^{-s}|$)
		\end{center}
		A priori using the fact, 
		\begin{center}
			$n^{-\sigma}\leq \int\limits_{n-1}^{n}{x^{-\sigma}}dx$,
		\end{center}
		performing summation on both sides yield,
		\begin{center}
			$\sum\limits_{N_{1}+1}^{N_{2}}n^{-\sigma}\leq  \int\limits_{N_{1}}^{N_{2}}{x^{-\sigma}}dx=\frac{1}{\sigma}(N_{1}^{-\sigma}-N_{2}^{-\sigma})$ $\longrightarrow 0$     as, $N_{1},N_{2}\rightarrow \infty$
		\end{center}
		Hence, $\sum\limits_{n=1}^{\infty}n^{-s}$ is convergent . (Applying \textit{Cauchy's Criterion} for convergence of series)
		
		Also,
		\begin{center} $\sum\limits_{p=prime}|p^{-s}|=\sum\limits_{p=prime}p^{-\sigma} \leq \sum\limits_{n=1}^{\infty}n^{-\sigma}$.
		\end{center}
		Hence, we conclude that, $\sum\limits_{p=prime}p^{-s}$ is \textit{absolutely convergent}.\par
		Applying Lemma \eqref{lemma4}, we obtain that, the product, $\prod\limits_{p=prime}(1-p^{-s})$ is absolutely convergent, hence convergent.
		Thus,
		\begin{center}
			$\prod\limits_{p=prime}\frac{1}{(1-p^{-s})}$ is also convergent.
		\end{center}
		To prove that, both sides of the given identity in this theorem are equal, we observe that,
		\begin{center}
			$\prod\limits_{p\leq N_{2}}\frac{1}{(1-\chi(p)p^{-s})}=\sum\limits_{n> N_{2}}\chi(n)n^{-s}+\sum\limits_{n\leq N_{2}}\chi(n)n^{-s}$
		\end{center}
		As, $N_{2}\longrightarrow \infty$,  
		\begin{center}
			R.H.S.    $=\sum\limits_{n\in \mathbb{N}}n^{-s}$ \hspace{40pt} (The sum being \textit{absolutely convergent})
		\end{center}
		Therefore, the product on the L.H.S. tends to $\prod\limits_{p=prime}\frac{1}{(1-p^{-s})}$, and hence, \textit{Euler's Formula} is established.
	\end{proof}
	\subsection{Riemann's Functional Equation for $\zeta(s)$}
	We define a particular form of \textit{Dirichlet Series} $F(x,s)$ as , 
	\begin{center}
		$	F(x,s)=\sum\limits_{n=1}^{\infty}\frac{{e}^{2\pi inx}}{{n}^{s}}$  , \hspace{10pt} where , $x\in \mathbb{R}$  \hspace{10pt} and ,\hspace{10pt} $Re(s)>1$ .
	\end{center}
	Important to observe that, $F(x,s)$ is a periodic function of $x$ ( also known as the \textit{periodic zeta function} ) with period $1$ \hspace{10pt} and , \hspace{10pt} $F(1,s)=\zeta(s)$ . 
	\begin{thm}
		(Hurwitz's Formula) We have,
		\begin{center}
			$\zeta(1-s,a)=\frac{\Gamma(s)}{{(2\pi)}^{s}}\{{e}^{-\pi is/2}F(a,s)+{e}^{\pi is/2}F(-a,s)\}$ .
		\end{center}
		For , $0<a\leq 1$  and , $Re(s)>1$ .
	\end{thm}
	\begin{rmk}
		For $a\neq 1$ , \textit{Hurwitz's Formula} is valid for $Re(s)>0$ .
	\end{rmk}
	Using \textit{Hurwitz's Formula} , we shall establish the functional equation for $\zeta(s)$ .	
	
	\begin{thm}
		The functional equation for the Riemann Zeta Function $\zeta(s)$ is given by ,
		\begin{center}
			$\zeta(1-s)=2{(2\pi)}^{-s}\Gamma(s)cos(\frac{\pi s}{2})\zeta(s)$
		\end{center}
		In other words ,
		\begin{center}
			$\zeta(s)=2{(2\pi)}^{s-1}\Gamma(1-s)cos(\frac{\pi s}{2})\zeta(1-s)$
		\end{center}
	\end{thm}
	\begin{proof}
		Putting $a=1$ in \textit{Hurwitz's Formula} , 
		\begin{center}
			$\zeta(1-s)=\frac{\Gamma(s)}{{(2\pi)}^{s}}\{{e}^{-\pi is/2}\zeta(s)+{e}^{\pi is/2}\zeta(s)\}=\frac{\Gamma(s)}{{(2\pi)}^{s}}2cos(\frac{\pi s}{2})\zeta(s)$ .
		\end{center}
		Which proves the first part of the theorem . Replacing $s$ by $1-s$ , we establish the equivalent definition of the functional equation for $\zeta(s)$ .
	\end{proof}
	\begin{rmk}
		Putting , $s=2n+1$ ,\hspace{10pt} $\forall\hspace{5pt} n\in \mathbb{N}$ ,  we obtain the \textit{trivial zeroes} of $\zeta(s)$ , consequently,
		\begin{center}
			$\zeta(-2n)=0$ \hspace{20pt}   $\forall\hspace{5pt} n\in \mathbb{N}$ .
		\end{center}
	\end{rmk}
	
	\section{Order Estimates for specific functions of $\zeta(s)$}
	\subsection{A Contour Integral representation of $\frac{\psi_{1}(x)}{x^{2}}$}
	
	\begin{thm}\label{thm8}
		If $c>1$ and $x\geq 1$, then,
		\begin{eqnarray}\label{10}
			\frac{\psi_{1}(x)}{x^{2}}=\frac{1}{2\pi i}\int\limits_{c-\infty i}^{c+\infty i}\frac{x^{s-1}}{s(s+1)}( -\frac{\zeta'(s)}{\zeta(s)})	ds.
		\end{eqnarray}
	\end{thm}
	\begin{proof}
		A priori from \eqref{9},
		\begin{center}
			$\frac{\psi_{1}(x)}{x}=\sum\limits_{n\leq x} (1-\frac{n}{x})\Lambda(n)$.
		\end{center}
		Apply Lemma \eqref{lemma3} with putting $k=1$, and $u=\frac{n}{x}$. In case when $n\leq x $, we obtain,
		\begin{align} \label{25}
			(1-\frac{n}{x})=\frac{1}{2\pi i}\int\limits_{c-\infty i}^{c+\infty i}\frac{(\frac{x}{n})^{s}}{s(s+1)}ds.
		\end{align}
		Multiply \eqref{25} by $\Lambda(n)$ and summing over all $n\leq x$,
		\begin{center}
			$\frac{\psi_{1}(x)}{x}=\sum\limits_{n\leq x}\frac{1}{2\pi i}\int\limits_{c-\infty i}^{c+\infty i}\frac{\Lambda(n)(\frac{x}{n})^{s}}{s(s+1)}ds=\sum\limits_{n=1}^{\infty}\frac{1}{2\pi i}\int\limits_{c-\infty i}^{c+\infty i}\frac{\Lambda(n)(\frac{x}{n})^{s}}{s(s+1)}ds$.
		\end{center}	
		Since the integral above vanishes if, $n>x$, we can derive that,
		\begin{eqnarray}\label{11}
			\frac{\psi_{1}(x)}{x}=\sum\limits_{n=1}^{\infty}\int\limits_{c-\infty i}^{c+\infty i}f_{n}(s)ds,\hspace{20pt}
			&\mbox{ where,} \hspace{10pt} 2\pi if_{n}(x)=\Lambda(n)\frac{(\frac{x}{n})^{s}}{s(s+1)}.
		\end{eqnarray}
		Suppose, we wish to interchange the sum and the integral in \eqref{11}. For this it suffices to prove that, the series,
		\begin{eqnarray}\label{12}
			\sum\limits_{n=1}^{\infty}\int\limits_{c-\infty i}^{c+\infty i}|f_{n}(s)|ds
		\end{eqnarray}
		is convergent. Important to observe that, the partial sum of the series given in \eqref{12} saitsfy the inequality,
		\begin{center}
			$	\sum\limits_{n=1}^{N}\int\limits_{c-\infty i}^{c+\infty i}\frac{\Lambda(n)(\frac{x}{n})^{c}}{|s||s+1|}ds=\sum\limits_{n=1}^{N}\frac{\Lambda(n)}{n^{c}}\int\limits_{c-\infty i}^{c+\infty i}\frac{x^{c}}{|s||s+1|}ds\leq A\sum\limits_{n=1}^{\infty}\frac{\Lambda(n)}{n^{c}}$,
		\end{center}	
		Where, $A$ is a constant. Therefore, we can conclude that, \eqref{12} converges. Therefore, we can interchange the sum and the integral in \eqref{11} to obtain,
		\begin{center}
			$\frac{\psi_{1}(x)}{x}=\int\limits_{c-\infty i}^{c+\infty i}\sum\limits_{n=1}^{\infty}f_{n}(s)ds=\frac{1}{2\pi i}\int\limits_{c-\infty i}^{c+\infty i}\frac{x^{s}}{s(s+1)}\sum\limits_{n=1}^{\infty}\frac{\Lambda(n)}{n^{s}}ds$
		\end{center}	
		\begin{center}
			$=\frac{1}{2\pi i}\int\limits_{c-\infty i}^{c+\infty i}\frac{x^{s}}{s(s+1)}\left( -\frac{\zeta'(s)}{\zeta(s)}\right)	ds$.
		\end{center}	
		Now, we divide both sides of the above identity to obtain the desired result.
	\end{proof}
	\begin{thm}\label{thm9}
		If $c>1$ and $x\geq 1$ we have,
		\begin{eqnarray}\label{13}
			\frac{\psi_{1}(x)}{x^{2}}-\frac{1}{2}\left(1-\frac{1}{x}\right)^{2}=\frac{1}{2\pi i}\int\limits_{c-\infty i}^{c+\infty i}x^{s-1}h(s)ds
		\end{eqnarray}
		where,
		\begin{eqnarray}\label{14}
			h(s)=\frac{1}{s(s+1)}\left(-\frac{\zeta'(s)}{\zeta(s)}-\frac{1}{s-1}\right).
		\end{eqnarray}
	\end{thm}
	\begin{proof}
		We use Lemma \eqref{lemma3} with $k=2$ to get,
		\begin{center}
			$\frac{1}{2}(1-\frac{1}{x})^{2}=\frac{1}{2\pi i}\int\limits_{c-\infty i}^{c+\infty i}\frac{x^{s}}{s(s+1)(s+2)}ds$,\hspace{20pt} where, $c>0$.
		\end{center}
		Replacing $s$ by $s-1$ in the integral ( keeping $c>1$ ),
		\begin{center}
			$\frac{1}{2}(1-\frac{1}{x})^{2}=\frac{1}{2\pi i}\int\limits_{c-\infty i}^{c+\infty i}\frac{x^{s-1}}{(s-1)(s)(s+1)}ds$,
		\end{center}
		Subtracting the above identity from the identity in Theorem \eqref{thm8}, we get our desired result.
	\end{proof}
	\begin{rmk}\label{rmk2}
		If we parametrize the path of integration by writing $s=c+it$, we obtain, $x^{s-1}=x^{c-1}x^{it}=x^{c-1}e^{it\log x}$. As a result, equation \eqref{13} becomes,
		\begin{eqnarray}\label{15}
			\frac{\psi_{1}(x)}{x^{2}}-\frac{1}{2}\left(1-\frac{1}{x}\right)^{2}=\frac{x^{c-1}}{2}\int\limits_{c-\infty i}^{c+\infty i}h(c+it)e^{itlog(x)}dt .
		\end{eqnarray}
	\end{rmk} 
	Our next task is to show that the R.H.S. of the identity \eqref{15} $\rightarrow 0$ as $x\rightarrow \infty$. As mentioned earlier, we first have to establish that, we can put $c=1$ in \eqref{15}. For this purpose, we need to study $\zeta(s)$ in the neighbourhood of the line, $\sigma =1$ ( taking, $s=\sigma + it$ in complex plane ).
	\subsection{Upper Bounds for $|\zeta(s)|$ and $|\zeta'(s)|$ near the line $\sigma =1$}
	
	In order to study $\zeta(s)$ near the line, $\sigma =1$, we use a particular representation of $\zeta(s)$ obtained from the following result.
	\begin{thm}\label{thm10}
		For any integer $N\geq 0$ and $\sigma >0$,
		\begin{eqnarray}\label{16}
			\zeta(s)=\sum\limits_{n=0}^{N}\frac{1}{n^{s}}+\frac{N^{1-s}}{s-1}-s\int\limits_{N}^{\infty}\frac{x-[x]}{x^{s+1}}	dx
		\end{eqnarray}
	\end{thm}
	\begin{proof}
		Apply \textit{Euler's Summation Formula} with $f(t)=\frac{1}{t^{s}}$ and with integers $x$ and $y$ to obtain,
		\begin{center}
			$\sum\limits_{y<n\leq x}\frac{1}{n^{s}}=\int\limits_{y}^{x}\frac{dt}{t^{s}}-s\int\limits_{y}^{x}\frac{t-[t]}{t^{s+1}}	dt$
		\end{center}
		Taking $y=N$ and, letting $x\rightarrow \infty$, keeping $\sigma >1$, this yields,
		\begin{center}
			$\sum\limits_{n=N+1}^{\infty}\frac{1}{n^{s}}=\int\limits_{N}^{\infty}\frac{dt}{t^{s}}-s\int\limits_{N}^{\infty}\frac{t-[t]}{t^{s+1}}	dt$
		\end{center}
		
		or,
		\begin{center}
			$\zeta(s)-\sum\limits_{n=0}^{N}\frac{1}{n^{s}}=\frac{N^{1-s}}{s-1}-s\int\limits_{N}^{\infty}\frac{x-[x]}{x^{s+1}}	dx$.
		\end{center}	
		This proves \eqref{16} for $\sigma >1$. For the case when $0<\delta\leq \sigma$, the integral is dominated by, $\int\limits_{N}^{\infty}\frac{1}{t^{\delta +1}}	dt$,  so it converges uniformly for $\sigma \geq\delta$ and hence represents an analytic function in the half-plane \textbf{$\sigma>0$}. Therefore, \eqref{16} holds for $\sigma >0$ by analytic continuation.
	\end{proof}
	\begin{rmk}\label{rmk3}
		We also deduce the formula of $\zeta'(s)$ obtained by differentiating each member of \eqref{16},
		\begin{align*}
			\hspace{20pt} \zeta'(s)=-\sum\limits_{n=1}^{N}\frac{\log n}{n^{s}}+s\int\limits_{N}^{\infty}\frac{(x-[x])\log x}{x^{s+1}}dx-\int\limits_{N}^{\infty}\frac{(x-[x])}{x^{s+1}}dx 
		\end{align*}
		\begin{align}\label{17}
			\hspace{220pt}-\frac{N^{1-s}\log N}{s-1}-\frac{N^{1-s}}{(s-1)^{2}} .
		\end{align}
	\end{rmk}
	The next theorem uses the relations derived to obtain an upper bound for both $|\zeta(s)|$ and $|\zeta'(s)|$.
	\begin{thm}\label{thm11}
		For every $A>0$, $\exists$ a constant $M$ ( depending on $A$ ) such that,
		\begin{eqnarray}\label{18}
			|\zeta(s)|\leq M\log t \hspace{20pt} &\mbox{ and,}\hspace{10pt}  |\zeta'(s)|\leq M\log^{2} t
		\end{eqnarray}
		$\forall$ $s$ with $\sigma \geq \frac{1}{2}$ satisfying,
		\begin{eqnarray}\label{19}
			\sigma>1-\frac{A}{\log t } \hspace{20pt} &\mbox{ and, }  t\geq e .
		\end{eqnarray}
	\end{thm}
	
	\begin{proof}
		If $\sigma \geq 2$, we have, $|\zeta (s)|\leq \zeta(2)$ and $|\zeta'(s)|\leq |\zeta'(2)|$ . Hence, the inequalities in \eqref{18} are trivially satisfied. \\
		Therefore, we can assume, $\sigma<2$ and $t\geq e$. Consequently, we get,
		\begin{center}
			$|s|\leq \sigma +t\leq 2+t< 2t$ and    $|s-1|\geq t$.
		\end{center}
		Hence, $\frac{1}{|s-1|}\leq \frac{1}{t}$. Estimating $|\zeta (s)|$ by using \eqref{16} yields,
		\begin{center}
			$|\zeta(s)|\leq \sum\limits_{n=1}^{N}\frac{1}{n^{\sigma}}+ 2t\int\limits_{N}^{\infty}\frac{1}{x^{\sigma+1}}dx+ \frac{N^{1-\sigma}}{t}=\sum\limits_{n=1}^{N}\frac{1}{n^{\sigma}}+\frac{2t}{\sigma N^{\alpha}}+ \frac{N^{1-\sigma}}{t}$.
		\end{center}
		A priori we force $N$ to depend on $t$ by choosing, $N=[t]$. Thus, $N\leq t<N+1$ and $\log n\leq \log t$ if $n\leq N$.\\
		Observe that, \eqref{19} implies, $1-\sigma< \frac{A}{\log t}$. Consequently,
		\begin{center}
			$\frac{1}{n^{\sigma}}=\frac{n^{1-\sigma}}{n}=\frac{1}{n}e^{(1-\sigma)\log n}<\frac{1}{n}e^{A\frac{\log n}{\log t}}\leq \frac{1}{n}e^{A}=O(\frac{1}{n})$.
		\end{center}
		Therefore,
		\begin{center}
			$	\frac{2t}{\sigma N^{\sigma}}\leq \frac{N+1}{N}$ and,  $\frac{N^{1-\sigma}}{t}=\frac{N}{t}.\frac{1}{N^{\sigma}}=O(\frac{1}{N})=O(1)$,
		\end{center}	
		So,
		\begin{center}
			$|\zeta(s)|=O\left(\sum\limits_{n=1}^{N}\frac{1}{n}\right)+O(1)=O(log(N))+O(1)=O(log(t))$.
		\end{center}	
		This proves the inequality for $|\zeta(s)|$ in \eqref{18}.\\\\
		To obtain the expression for $|\zeta'(s)|$ we apply similar arguments on the equation \eqref{17}. It is significant to note that, an extra factor $\log N$ appears on R.H.S. Although the following estimate, $\log N=O(\log t)$ gives,   $|\zeta'(s)|=O(\log^{2} t)$ in the specified region.
		
	\end{proof}
	
	\subsection{The Non-Vanishing of $\zeta(s)$ on the line, $\sigma=1$}
	
	As evident from the title, in this section, we indeed wish to establish that, $\zeta(1+it)\neq 0$,   $\forall t\in \mathbb{R}$. The proof primarily focuses on an inequality.
	\begin{thm}\label{thm12}
		Given that the \textit{Dirichlet Series} $F(s)=\sum \frac{f(n)}{n^{s}}$ corresponding to an arithmetic function  $f(n)$ be absolutely convergent for $\sigma>\sigma_{a}$. We further assume, $f(1)\neq 0$. If $F(s)\neq 0$ for $\sigma>\sigma_{0}>\sigma_{a}$, then, for $\sigma>\sigma_{0}$, 
		\begin{center}
			$F(s)=e^{G(s)}$
		\end{center}
		with, 
		\begin{center}
			$G(s)=\log f(1)+\sum\limits_{n=2}^{\infty}\frac{(f'*f^{-1})(n)}{\log n}\frac{1}{n^{s}}$;
		\end{center}
		where, $f^{-1}$ is the \textit{Dirichlet Inverse }of $n$ and, $f'(n)=f(n)\log n$.
		
	\end{thm}
	\begin{rmk}\label{rmk5}
		For complex $z\neq 0$, $\log z$ denotes that branch of the logarithm which is real when $z>0 $.
	\end{rmk} 
	\begin{proof}
		Since, $F(s)\neq 0$, we can write, 	$F(s)=e^{G(s)}$ for some function $G(s)$, which is analytic for $\sigma>\sigma_{0}$. Differentiation gives us,
		\begin{center}
			$F'(s)=e^{G(s)}G'(s)=F(s)G'(s)$.
		\end{center} 
		Therefore, we get,    $G'(s)=\frac{F'(s)}{F(s)}$. Although, we have,
		\begin{center}
			$F'(s)=-\sum\limits_{n=1}^{\infty}\frac{f(n)\log n}{n^{s}}=-\sum\limits_{n=1}^{\infty}\frac{f'(n)}{n^{s}}$  \hspace{20pt} and,          \hspace{10pt} $\frac{1}{F(s)}=\sum\limits_{n=1}^{\infty}\frac{f^{-1}(n)}{n^{s}}$
		\end{center}
		Hence,   $G'(s)=\frac{F'(s)}{F(s)}=-\sum\limits_{n=2}^{\infty}\frac{(f'*f^{-1})(n)}{n^{s}}$.\\\\
		\textit{Integration} gives,
		\begin{center}
			$G(s)=C+\sum\limits_{n=2}^{\infty}\frac{(f'*f^{-1})(n)}{n^{s}\log n}$
		\end{center}
		Where, $C$ is a constant. As, $\sigma\rightarrow +\infty$, we obtain,
		\begin{center}
			$\lim\limits_{\sigma\rightarrow +\infty}G(\sigma +it)=C$.
		\end{center}
		Therefore, 
		\begin{center}
			$f(1)=\lim\limits_{\sigma\rightarrow +\infty}F(\sigma +it)=e^{C}$.
		\end{center}
		Thus, $C=\log f(1)$ and this completes the proof . 
	\end{proof}
	\begin{rmk}
		It can also be deduced from the proof that, the series for $G(s)$ converges absolutely if $\sigma>\sigma_{0}$.
	\end{rmk}
	\begin{cor}\label{cor1}
		For the Riemann Zeta Function $\zeta(s)$, we have,
		\begin{center}
			$\zeta(s)=e^{G(s)}$, \hspace{20pt} for, \hspace{10pt}$\sigma>1$ 
		\end{center}
		where, 
		\begin{center}
			$G(s)=\sum\limits_{n=2}^{\infty}\frac{\Lambda(n)}{\log n }n^{-s}$.
		\end{center}
	\end{cor}
	\begin{proof}
		Using $f(n)=1$, and applying Theorem \eqref{thm12}, we obtain, $f'(n)=\log n$ and $f^{-1}(n)=\mu(n)$,\\
		Therefore,
		\begin{center}
			$(f'*f^{-1})(n)=\sum\limits_{d|n}\log d\mu(\frac{n}{d})=\Lambda (n)$.
		\end{center}
		Hence, for $\sigma >1$, we obtain the result.
	\end{proof}	
	As a direct application of \textit{Corollary} \eqref{cor1}, we prove our main result in this section.	
	\begin{thm}\label{thm13}
		For $\sigma>1$, 
		\begin{eqnarray}\label{20}
			\zeta^{3}(\sigma)|\zeta(\sigma+it)|^{4}|\zeta(\sigma+2it)|\geq 1.
		\end{eqnarray}
	\end{thm}
	\begin{proof}
		A priori from Corollary \eqref{cor1}, we have, for $\sigma>1$,
		\begin{center}
			$\zeta(s)=e^{G(s)}$
		\end{center}	
		where,
		\begin{center}
			$G(s)=\sum\limits_{n=2}^{\infty}\frac{\Lambda(n)}{log(n)n^{s}}=\sum\limits_{p}\sum\limits_{m=1}^{\infty}\frac{1}{mp^{ms}}$  .
		\end{center}	
		In other words,
		\begin{center}
			$\zeta(s)=e^{\{\sum\limits_{p}\sum\limits_{m=1}^{\infty}\frac{1}{mp^{ms}}\}}=e^{\{\sum\limits_{p}\sum\limits_{m=1}^{\infty}\frac{e^{-imtlog(p)}}{mp^{m\sigma}}\}}$.
		\end{center}	
		from which, we obtain,
		\begin{center}
			$|\zeta(s)|=e^{\{\sum\limits_{p}\sum\limits_{m=1}^{\infty}\frac{cos(mtlog(p))}{mp^{m\sigma}}\}}$.
		\end{center}	
		Successive application of this formula with $s=\sigma$, $s=\sigma+it$, $s=\sigma+2it$ implies,
		\begin{center}
			$\zeta^{3}(\sigma)|\zeta(\sigma+it)|^{4}|\zeta(\sigma+2it)|=e^{\{\sum\limits_{p}\sum\limits_{m=1}^{\infty}\frac{3+4cos(mtlog(p))+cos(2mtlog(p))}{mp^{m\sigma}}\}}$
		\end{center}
		Now the following trigonometric identity,
		\begin{center}
			$3+4cos\theta+cos2\theta=3+4cos\theta+4cos^{2}2\theta-1=2(1+cos\theta)^{2}\geq 0$.
		\end{center} 	
		helps us assert that, each term in the last infinite series is non-negative. Consequently, we obtain \eqref{20}.
	\end{proof}
	We thus conclude the following.
	\begin{thm}\label{thm14}
		$\zeta(1+it)\neq 0$ for every real $t$.
	\end{thm}
	\begin{proof}
		It is sufficient to consider the case when $t\neq 0$. A priori from \eqref{20} we deduce, 
		\begin{eqnarray}\label{21}
			{(\sigma-1)}^{3} {\zeta(\sigma)}^{3}|\frac{\zeta(\sigma+it)}{\sigma-1}|^{4}|\zeta(\sigma+2it)|\geq \frac{1}{\sigma-1}
		\end{eqnarray}
		It is easy to verify that, the above identity is valid for $\sigma>1$. Now, as $\sigma\rightarrow 1^{+}$ in \eqref{21}, we observe that, the first factor approaches $1$    [ Since, $\zeta(s)$ has residue $1$ at the pole $s=1$ ]. The third factor tends to $|\zeta(1+2it)|$. If, $|\zeta(1+it)|=0$, then we could have written the middle factor as, 
		\begin{center}
			$	|\frac{\zeta(\sigma+it)-\zeta(1+it)}{\sigma-1}|^{4}\rightarrow |\zeta(1+it)|^{4}$ as $\sigma\rightarrow 1^{+}$.
		\end{center}	
		Thus, for some, $t\neq 0$, if we had, $\zeta(1+it)=0$, the L.H.S. of the equation\eqref{21} would approach the limit, 
		\begin{center}
			$|\zeta'(1+it)|^{4}|\zeta(1+2it)|$ \hspace{20pt} as \hspace{10pt}  $\sigma\rightarrow 1^{+}$.
		\end{center} 
		But, the R.H.S. tends to $\infty$ as $\sigma\rightarrow 1^{+}$ and this yields a contradiction.
	\end{proof}
	
	\subsection{Inequalities for $|\frac{1}{\zeta(s)}|$ and $|\frac{\zeta'(s)}{\zeta(s)}|$}
	
	Applying Theorem \eqref{thm13}, we in fact comment on the bounds of the functions, $|\frac{1}{\zeta(s)}|$ and $|\frac{\zeta'(s)}{\zeta(s)}|$.
	\begin{thm}\label{thm15}
		$\exists$ a constant $M>0$ such that,
		\begin{center}
			$|\frac{1}{\zeta(s)}|<M\log^{7}t$ \hspace{10pt}and,\hspace{10pt}$|\frac{\zeta'(s)}{\zeta(s)}|<M\log^{9}t$  
		\end{center}
		whenever, $\sigma\geq 1$ and,  $t\geq e$.
	\end{thm}
	\begin{proof}
		For $\sigma\geq 2$, we have,
		\begin{center}
			$|\frac{1}{\zeta(s)}|=|\sum\limits_{n=1}^{\infty}\frac{\mu(n)}{n^{s}}|\leq \sum\limits_{n=1}^{\infty}\frac{1}{n^{2}}\leq \zeta(2)$ \hspace{10pt} and, \hspace{10pt} $|\frac{\zeta'(s)}{\zeta(s)}|\leq \sum\limits_{n=1}^{\infty}\frac{\Lambda(n)}{n^{2}}$,
		\end{center}
		Therefore, the inequalities hold trivially provided, $\sigma\geq 2$. Suppose, $1\leq \sigma\leq 2$ and $t\geq e$. Rewriting inequality \eqref{20}, we get,
		\begin{center}
			$\frac{1}{|\zeta(\sigma+it)|}\leq (\zeta(\sigma))^{\frac{3}{4}}|\zeta(\sigma+2it)|^{\frac{1}{4}}$.
		\end{center}	
		Now, observe that, $(\sigma-1)\zeta(\sigma)$ is bounded in the interval $1\leq \sigma\leq 2$, say,  $(\sigma-1)\zeta(\sigma)\leq M$, where, $M$ is an absolute constant. Then,
		\begin{center}
			$\zeta(\sigma)\leq \frac{M}{\sigma-1}$ \hspace{20pt} if, \hspace{10pt} $1< \sigma\leq 2$ .
		\end{center}	
		Furthermore, $\zeta(\sigma+2it)=O(logt)$ if, $1\leq \sigma\leq 2$  [ By Theorem \eqref{thm11} ]. Hence, for $1< \sigma\leq 2$, we have,
		\begin{center}
			$\frac{1}{|\zeta(\sigma+it)|}\leq \frac{M^{\frac{3}{4}}(logt)^{\frac{1}{4}}}{(\sigma-1)^{\frac{3}{4}}}=\frac{A(logt)^{\frac{1}{4}}}{(\sigma-1)^{\frac{3}{4}}}$,
		\end{center}	
		where, $A$ is an absolute constant. Therefore, for some constant $B>0$, 
		\begin{eqnarray}\label{22}
			|\zeta(\sigma+it)|>\frac{B(\sigma-1)^{\frac{3}{4}}}{(logt)^{\frac{1}{4}}} \hspace{40pt} &\mbox{ if,} 1<\sigma\leq 2 &\mbox{ and  } t\geq e.
		\end{eqnarray}	
		This also holds trivially for $\sigma=1$. Let, $\alpha$ be any number satisfying, $1<\alpha<2$.  Then, if $1\leq \sigma\leq \alpha$, $t\geq e$,  we may use Theorem \eqref{thm11} to write,
		\begin{center}
			$|\zeta(\sigma+it)-\zeta(\alpha+it)|\leq \int\limits_{\sigma}^{\alpha}|\zeta'(u+it)|du\leq (\alpha-\sigma)Mlog^{2}t$
		\end{center}	
		\begin{center}
			$\leq (\alpha-1)Mlog^{2}t$.
		\end{center}	
		Hence, by \textit{triangle inequality},
		\begin{center}
			$|\zeta(\sigma+it)|	\geq |\zeta(\alpha+it)|-|\zeta(\sigma+it)-\zeta(\alpha+it)|\geq |\zeta(\alpha+it)|-(\alpha-1)Mlog^{2}t\geq \frac{B(\sigma-1)^{\frac{3}{4}}}{(logt)^{\frac{1}{4}}}- (\alpha-1)Mlog^{2}t$.
		\end{center}
		provided, $1\leq \sigma\leq \alpha$,  and using \eqref{22}, it also holds for $\alpha\leq \sigma\leq 2$, since, $(\sigma-1)^{\frac{3}{4}}\geq (\alpha-1)^{\frac{3}{4}}$. In other words, $1\leq \sigma\leq 2$ and, $t\geq e$, we have the inequality,
		\begin{center}
			$|\zeta(\sigma+it)|\geq \frac{B(\sigma-1)^{\frac{3}{4}}}{(logt)^{\frac{1}{4}}}-(\alpha-1)Mlog^{2}t$,
		\end{center} 
		for any $\alpha$ satisfying $1<\alpha<2$. Now, we make $\alpha$ depend on $t$ and choose $\alpha$ so the first term on R.H.S. is twice the second. This requires,
		\begin{center}
			$\alpha=1+(\frac{B}{2M})^{4}\frac{1}{(logt)^{9}}$.
		\end{center}	
		Clearly, $\alpha>1$ and also $\alpha <2$ if, $t\geq t_{0}$ for some $t_{0}$. Thus, if , $t\geq t_{0}$ and $1\leq \sigma\leq 2$, we then have,
		\begin{center}
			$|\zeta(\sigma+it)|\geq (\alpha-1)Mlog^{2}t=\frac{C}{(logt)^{7}}$. 
		\end{center}	
		This inequality should also be true perhaps for a different value of $C$ if, $e\leq t\leq t_{0}$.\\\\
		Hence, 
		\begin{center}
			$|\zeta(s)|\geq \frac{C}{(logt)^{7}}$, \hspace{40pt}  $\forall$  $\sigma\geq 1$ and, $t\geq e$,
		\end{center}	
		which gives us a corresponding upper bound for $|\frac{1}{\zeta(s)}|$.\\\\
		To get the inequality for $|\frac{\zeta'(s)}{\zeta(s)}|$, we apply Theorem \eqref{thm11} to obtain an extra factor $log^{2}t$.
	\end{proof}
	\section{Analytic Proof of PNT}	
	Before giving our readers the detailed proof of the\textit{"Prime Number Theorem"}, we shall introduce the following result from complex function theory and another theorem which will be required later on in the proof of the main theorem.
	\begin{lemma}\label{lemma5}
		If $f(s)$ has a pole of order $k$ at $s=\alpha$, then the quotient, $\frac{f'(s)}{f(s)}$ has a first order pole at $s=\alpha$ with residue $-k$.
	\end{lemma}
	\begin{proof}
		We have,  $f(s)=\frac{g(s)}{(s-\alpha)^{k}}$,  where, $g$ is analytic at $\alpha$ and $g(\alpha)\neq 0$. Hence, for all $s$ in a neighbourhood of $\alpha$, 
		\begin{center}
			$f'(s)=\frac{g'(s)}{(s-\alpha)^{k}}-\frac{kg(s)}{(s-\alpha)^{k+1}}=\frac{g(s)}{(s-\alpha)^{k}}\{\frac{-k}{(s-\alpha)}+\frac{g'(s)}{g(s)}\}$.
		\end{center}	
		Consequently,
		\begin{center}
			$\frac{f'(s)}{f(s)}=\{\frac{-k}{(s-\alpha)}+\frac{g'(s)}{g(s)}\}$.
		\end{center}	
		Which proves the lemma, since $\frac{g'(s)}{g(s)}$ is analytic at $\alpha$.
	\end{proof}
	\begin{thm}\label{thm16}
		The function,
		\begin{center}
			$F(s)=-\frac{\zeta'(s)}{\zeta(s)}-\frac{1}{s-1}$
		\end{center}
		is analytic at $s=1$.
	\end{thm}
	\begin{proof}
		A priori it follows from the Lemma \eqref{lemma5},  $-\frac{\zeta'(s)}{\zeta(s)}$ has a first order pole at $1$ with residue $1$, as does $\frac{1}{s-1}$. Hence their difference is analytic at $s=1$, and the theorem is proved.
	\end{proof}
	Having proven all the necessary results, we can indeed reformulate the statement we intend to establish in order to prove the \textit{Prime Number Theorem} for our convenience.
	\begin{thm}\label{thm17}
		For, $x\geq 1$, we have,
		\begin{center}
			$\frac{\psi_{1}(x)}{x^{2}}-\frac{1}{2}(1-\frac{1}{x})^{2}=\frac{1}{2\pi}\int\limits_{-\infty}^{\infty}h(1+it)e^{it.logx}dt$,
		\end{center} 	
		where,  the integral,  $\int\limits_{-\infty}^{\infty}|h(1+it)|dt $ converges.\par
		Therefore, the \textbf{Riemann-Lebesgue Lemma} (cf. Theorem \eqref{thm4}) implies,
		\begin{eqnarray}\label{23}
			\psi_{1}(x)\sim \frac{x^{2}}{2}
		\end{eqnarray}
		and hence,
		\begin{center}
			$\psi(x)\sim x$ \hspace{20pt} as,  $x\rightarrow \infty$.
		\end{center}
		which is equivalent to the original statement of "Prime Number Theorem" \eqref{thm3}.
	\end{thm}	
	\begin{proof}
		In Theorem \eqref{thm9}, its already proved that,
		\begin{center}
			$\frac{\psi_{1}(x)}{x^{2}}-\frac{1}{2}(1-\frac{1}{x})^{2}=\frac{1}{2\pi i}\int\limits_{c-\infty i}^{c+\infty i}x^{s-1}h(s)ds$,
		\end{center}	
		for $c>1$ and $x\geq 1$, where,
		\begin{center}
			$h(s)=\frac{1}{s(s+1)}\{-\frac{\zeta'(s)}{\zeta(s)}-\frac{1}{s-1}\}$.
		\end{center}
		Our first objective is to show that, we can shift the path of integration to the line, $\sigma =1$. To do this, we apply \textbf{Cauchy's Theorem} (cf. Theorem \eqref{thm5}) to the rectangle $R$ shown in the \textit{Figure}  \eqref{fig1}.\par
		\begin{figure}[hbt!]
			\centering
			\begin{tikzpicture}
				\draw[thick,<-] (-2,0) -- (-0.2,0) node[anchor=north] {O};
				\draw[thick,->] (-0.2,0) -- (8,0) node[anchor=north] {Re(s)};
				\draw[thick,<->] (0,-5) -- (0,5) node[anchor=east] {Im(s)};	
				\draw[line width=2pt,red, dashed] (2,-4) -- (5,-4) -- (5,4) -- (2,4) -- (2,-4);
				\draw[line width=2pt,black] (2,-4) -- (2,-4) node[anchor=north] {$1-iT$};
				\draw[line width=2pt,black] (5,-4) -- (5,-4) node[anchor=north] {$c-iT$};
				\draw[line width=2pt,black] (5.2,0) -- (5.2,0) node[anchor=north] {$c$};
				\draw[line width=2pt,black] (5,4.5) -- (5,4.5) node[anchor=north] {$c+iT$};
				\draw[line width=2pt,black] (2,4.5) -- (2,4.5) node[anchor=north] {$1+iT$};
				\draw[line width=2pt,black] (1.8,0) -- (1.8,0) node[anchor=north] {$1$};
				\draw[thick,->] (2.5,-3.5) -- (4.5,-3.5) node[anchor=north] { };
				\draw[thick,->] (4.5,0.5) -- (4.5,3) node[anchor=east] { };
				\draw[thick,->] (4.5,3.5) -- (2.5,3.5) node[anchor=north] { };
				\draw[thick,->] (2.5,-0.5) -- (2.5,-3) node[anchor=east] { };
			\end{tikzpicture}
			\caption{}
			\label{fig1}
		\end{figure}
		
		The integral of $x^{s-1}h(s)$ around $R$ is $0$, since the integrand is analytic inside and on $R$. Next, we establish that, the integrals along the horizontal segments tends to $0$ as $T\rightarrow \infty$.
		Since the integrand has the same absolute values at the conjugate points, it suffices to consider only the upper segment, $t=T$. On this segment, we have, the estimates,
		\begin{center}
			$|\frac{1}{s(s+1)}|\leq \frac{1}{T^{2}}$ and,  $|\frac{1}{s(s+1)(s-1)}|\leq \frac{1}{T^{3}}\leq \frac{1}{T^{2}}$.
		\end{center}	
		Therefore, $\exists$ a constant $M$ such that, $|\frac{\zeta'(s)}{\zeta(s)}|\leq M(\log T)^{9}$ if,    $\sigma\geq 1$ and $t\geq e$. Hence, if $T\geq e$, then, we shall have,
		\begin{center}
			$|h(s)|\leq \frac{M(\log T)^{9}}{T^{2}}$
		\end{center}	
		
		so that,
		\begin{center}
			$|\int\limits_{1+iT}^{c+iT}x^{s-1}h(s)ds|\leq \int\limits_{1}^{c}x^{c-1}\frac{M(\log T)^{9}}{T^{2}}d\sigma=Mx^{c-1}\frac{(\log T)^{9}}{T^{2}}(c-1)$.
		\end{center}
		Similar approach follows for the lower segment, $t=-T$ in order to obtain the corresponding estimate ,
		\begin{center}
			$|\int\limits_{1-iT}^{c-iT}x^{s-1}h(s)ds|\leq \int\limits_{1}^{c}x^{c-1}\frac{M(\log T)^{9}}{T^{2}}d\sigma=Mx^{c-1}\frac{(\log T)^{9}}{T^{2}}(c-1)$.
		\end{center}

		Therefore, the integrals along the horizontal segments tend to $0$ as $T\rightarrow \infty$ , and,
		\begin{center}
			$\int\limits_{c-\infty i}^{c+\infty i}x^{s-1}h(s)ds=\int\limits_{1-\infty i}^{1+\infty i}x^{s-1}h(s)ds$.
		\end{center}	
		On the line $\sigma=1$, we put $s=1+it$ to get,
		\begin{center}
			$\frac{1}{2\pi i}\int\limits_{1-\infty i}^{1+\infty i}x^{s-1}h(s)ds=\frac{1}{2\pi}\int\limits_{-\infty}^{\infty}h(1+it)e^{it\log x}dt$.
		\end{center}	
		Note that,
		\begin{center}
			$\int\limits_{-\infty}^{\infty}|h(1+it)| dt=\int\limits_{-e}^{e}|h(1+it)| dt+\int\limits_{e}^{\infty}|h(1+it)| dt+\int\limits_{-\infty}^{-e}|h(1+it)| dt$.
		\end{center}	
		Now, in the integral, $\int\limits_{e}^{\infty}|h(1+it)| dt$, we observe that,
		\begin{center}
			$|h(1+it)|\leq \frac{M(\log t)^{9}}{t^{2}}$
		\end{center}	
		
		Hence, $\int\limits_{e}^{\infty}|h(1+it)|dt$ converges.	Similarly,  $\int\limits_{-\infty}^{-e}|h(1+it)|dt$ converges, therefore,  $\int\limits_{-\infty}^{\infty}|h(1+it)|dt$ converges. Thus, we may apply the \textbf{Riemann-Lebesgue Lemma} (cf. Theorem \eqref{thm4}) to obtain, 
		\begin{center}
			$\psi_{1}(x)\sim \frac{x^{2}}{2}$.
		\end{center}
		By Theorem \eqref{thm7}, the above result implies that, 
		\begin{center}
			$\psi(x)\sim x$ as,  $x\rightarrow \infty$.
		\end{center}	
		This proves the \textit{\textbf{Prime Number Theorem}}.
		
	\end{proof}

	\vspace{80pt}
	\section*{Acknowledgments}
I'll always be grateful to \textbf{Prof. Baskar Balasubramanyam} ( Associate Professor, Department of Mathematics, IISER Pune, India ), whose unconditional support and guidance helped me in understanding the topic and developing interest towards Analytic Number Theory.
	\section*{Statements and Declarations}
	\subsection*{Conflicts of Interest Statement}
	I as the author of this article declare no conflicts of interest.
	\subsection*{Data Availability Statement}
	I as the sole author of this article confirm that the data supporting the findings of this study are available within the article [and/or] its supplementary materials.

	\bibliographystyle{elsarticle-num}

	\end{document}